\documentclass[final,leqno,onefignum,onetabnum]{siamltex1213}
\usepackage[utf8]{inputenc}
\usepackage[english]{babel}
\usepackage{amsmath}
\usepackage{amssymb}
\usepackage{color}
\usepackage[usenames]{xcolor}
\usepackage{fancyhdr}

\usepackage{lipsum}

\usepackage{hyperref}
\hypersetup{
    unicode=false,          
    pdftoolbar=true,        
    pdfmenubar=true,        
    pdffitwindow=false,     
    pdfstartview={FitH},    
    pdftitle={My title},    
    pdfauthor={Author},     
    pdfsubject={Subject},   
    pdfnewwindow=true,      
    colorlinks=true,       
    linkcolor=blue,          
    citecolor=blue,        
    filecolor=magenta,      
    urlcolor=cyan           
}

\usepackage{graphicx}
\usepackage{epsfig}
\usepackage{epsfig,exscale,graphicx,graphics}
\usepackage{tikz}
\usetikzlibrary{arrows}

\usepackage[top=1in, bottom=1.25in, left=1.25in, right=1.25in]{geometry}

\title{Persistence of Network Synchronization under Nonidentical Coupling Functions}
\author{Daniel M. N. Maia$^{1,2}$
, 
 Elbert E. N. Macau$^{1}$
 , and 
  Tiago Pereira$^{3,4}$
  \\
$^{1}$ Instituto Nacional de Pesquisas Espaciais, S\~ao Jos\'e dos Campos, Brazil \\ 
$^{2}$ Humboldt-Univesit\"at zu Berlin, Germany\\
$^{3}$ Department of Mathematics, Imperial College London, UK \\
$^{4}$ Universidade de S\~ao Paulo, S\~ao Carlos, Brazil.
}

\newtheorem{assumption}{Assumption}
\newtheorem{remark}{Remark}

\newcommand{\R}{\mathbb{R}}

\begin{document}
\maketitle

\begin{abstract}
We investigate the persistence of synchronization in networks of diffusively coupled oscillators  when the coupling functions are nonidentical. Under mild conditions, we uncover the influence of the  network interaction structure on the mismatches of the coupling function. We show that Erd\"os-R\'enyi random graphs support large perturbations in the coupling function. In contrast scale-free graphs do not allow large perturbations in the coupling function, that is, as the network size $n$ goes to infinity  it forces the coupling functions to be identical. 
\end{abstract}


\section{Introduction}


Recent empirical studies of real complex systems have led to a deep understanding of the structure of networks and of the coupling function. In particular, experimental findings revealed that the interaction between diffusively coupled oscillators can be mediated  by different coupling functions \cite{Aneta2015}. For instance, the cardio-respiratory coupling function can be decomposed into a number of independent functions of a time-varying nature \cite{Stankovski2012}.
Moreover, different time-varying coupling functions can be used in the context of networks with time-varying topology \cite{Stiwell}. The mathematical theory for synchronization in the presence of nonidentical coupling function remains elusive \cite{Pecora,Pereira2014,Pereira2014_1}. A natural question in this context concerns the stability of the synchronized motion for nonidentical coupling functions. 

We  provide some conditions for the persistence of synchronized motion when the coupling functions are nonidentical.We consider undirected, simple and connected networks, see Ref. \cite{Bondy} for details. The dynamics of the  $n$ 
identical oscillators diffusively coupled is described as
\begin{equation}\label{eq:main}
\dot x_i = f(x_i) + \alpha \sum_{j=1}^n A_{ij} H_{ij}(t,x_j - x_i), \quad i=1,\cdots,n
\end{equation}
where $\alpha>0$ is the overall coupling strength, the adjacency matrix $A = (A_{ij})_{i,j=1}^n$ describes the interaction structure of the network ($A_{ij} = 1$ if $i$ is connected to $j$ and $0$ otherwise), the function $f : \mathbb{R}^q \rightarrow \mathbb{R}^q$ describes the isolated node dynamics, and the family of coupling functions $H_{ij} : \mathbb{R}_+ \times \mathbb{R}^q \rightarrow \mathbb{R}^q$ describes an interaction akin to diffusion between nodes $i,j=1,\cdots,n$.  

Consider the mismatches between coupling functions $H_{ij} - H_{kl}$. Our main contribution provides sufficient conditions on the mismatches to guarantee stable synchronization.  Loosely speaking, our results show that \\ 

\begin{itemize}
\item[]-- Erd\"os-R\'enyi networks support large mismatches;
\item[]-- Scale-Free networks forces the mismatches to converge to zero as $n \rightarrow \infty$. \\
\end{itemize} 

In order words, the heterogeneity in the degrees $g_i = \sum_{j=1}^n A_{ij}$ determines the  mismatches size.  
The precise statement  of our results can be found in Sec. \ref{MR} and the numerical results in Sec. \ref{Ill}.

\section{Main Results}\label{MR}

Since our main goal is to study the effect of the coupling function on the synchronization behavior, 
we keep the vector field $f$ identical for all nodes. If the vector field is nonidentical we can use the approach developed in \cite{Pereira2014} to obtain the collective behavior of the model. 

To fix notation, throughout the manuscript we use the norm  $\|  x_i \| = \max_{j}{|x_{ij}|}$, for $x_i = (x_{i1}, \cdots, x_{iq})$.  When dealing with operators we use the induced operator norm. For instance, if $C = (C_{ij})_{i,j=1}^n$ is a matrix
then $\Vert C \Vert = \max_i \sum_{j}|C_{ij}|$. 
When using the Euclidean norm, we represent it as $\Vert\cdot\Vert_2$.
To state our results we proceed with some assumptions on the vector field $f$ and coupling functions $H_{ij}$.

\begin{assumption}\label{assumption_vector_field}
The function $f$ is continuous differentiable  and there exists an inflowing invariant open ball $U \subset \mathbb{R}^q$  with
$$
\| D f (x) \| \le \varrho \, \, \, \, \mbox{  for all  } x \in U  
$$
for some $\varrho > 0$.
\end{assumption}

\begin{assumption}[Coupling Function Perturbation]
We assume that every coupling function can be factorized as
$$
H_{ij}(t,x) = H(x) + \widetilde{P}_{ij}(t,x).
$$
satisfying
\begin{enumerate}\label{assumption_coupling}
\item[2.1] $H_{ij}(t,0) = 0$.
\item[2.2] $H$ is differentiable and $DH(0) = \Gamma$ has eigenvalues $\gamma_i$ satisfying 
\begin{equation}\label{gamma}
\gamma =\gamma(\Gamma)= \min_{1\leq i\leq q} \Re(\gamma_i) >0.
\end{equation}
\item[2.3]\label{assump:23} The perturbations $\tilde{P}_{ij}:\mathbb{R}_{+}\times \R^q \to\mathbb{R}^q$ are continuous  matrices 
satisfying
\begin{eqnarray}
\widetilde{P}_{ij}(t,x) &=& P_{ij}(t) x \nonumber \\
\sup_{t > 0, x \in U} \| P_{ij}(t) \|  &\leq& \delta \quad \mbox{for all } i,j \in\{1,\cdots,n\}.  \label{delta}
\end{eqnarray}
\end{enumerate}
\end{assumption}

Because of the diffusive nature of the coupling, if all oscillators start with the same initial condition, then the coupling term vanishes identically. This ensures that the globally synchronized state 
$x_i(t) = s(t)$ for all $i = 1,2, \dots, n$ 
is an invariant state for all coupling strengths $\alpha$ and all choices of coupling functions $H_{ij}$. 
We call the subset 
$$
\mathcal{S}:=\{ {x}_i \in U\subset \mathbb{R}^q \mbox{ for }  i \in \{1, \dots, n\} : x_1=\cdots = x_n \}
$$
the synchronization manifold. The local stability of $\mathcal{S}$ is determined by the spectral properties of the 
combinatorial Laplacian $L$. Consider the diagonal matrix $D = \operatorname{diag}(g_1,\cdots,g_n)$ where
again $g_i = \sum_{j=1}^n A_{ij}$
denotes the degree of the vertex $i$. Then the Laplacian matrix reads $L = D - A$. As we are considering symmetric networks, the eigenvalues of $L$ are all real and they  can be arranged of the form 
$$0=\lambda_1<\lambda_2\le \cdots \le \lambda_n.$$
The  second eigenvalue $\lambda_2$, known as algebraic connectivity of the graph, plays an important role in the stability analysis of the synchronization manifold. Our main result determines the perturbation size $\delta$ 
(in Assumption \ref{assumption_coupling}) in terms of the network structure. All proofs of the following results are placed at the Appendix \ref{appendix_proofs}. \\

\begin{theorem}[Persistence]\label{theo:persistence}
Consider the model in Eq. \eqref{eq:main} satisfying Assumptions \ref{assumption_vector_field} and  \ref{assumption_coupling} on a connected network. Then, there exists constants $\eta = \eta(f,\Gamma)$ and $K = K(\Gamma)$ such that for all coupling strengths satisfying
\begin{equation}\label{eq:condition1}
\alpha > \frac{\eta}{\lambda_2 \gamma}
\end{equation} 
and perturbations of the coupling function satisfying
\begin{equation}\label{eq:persistentcondition}
\delta < \frac{\lambda_2 \gamma - \eta/\alpha}{ K\Vert L\Vert},
\end{equation}
where $\gamma$ is given by Eq. (\ref{gamma}) and
$\lambda_2 = \lambda_2(L)$ is the algebraic connectivity,
the synchronization manifold is locally exponentially stable. That is, there exist constants $\rho_0 >0$ and $C>0$
such that
if $\| x_j(t_0) - x_i(t_0) \|_2 \le \rho_0$ for all $i,j=1,\cdots,n$, then 
$$
\|  x_j(t) - x_i(t) \|_2 \le Ce^{-\nu (t-t_0)}\| x_j(t_0) - x_i(t_0) \|_2,
$$    
for all $t\geq t_0$ and all $i,j=1,\cdots,n$, where
\begin{equation}\label{eq_nu}
\nu  = \alpha\lambda_2 \gamma -\eta - \delta K \| L \|>0.
\end{equation}
\end{theorem}

One challenge to be overcome in the proof of the above result is to show that the constant $K$ is indenpendend of the network size. Hence, the 
network contribution can be factored in terms of the spectral gap $\lambda_2$ and the spectral radius via $\| L\|$. This is only true for undirected networks. For digraphs, $K$ can depend badly on the network size $n$ and no persistence result may be possible.

Notice that the parameter $\nu$ in Eq. \eqref{eq_nu} provides the decaying rate towards synchronization.
The perturbation slows down the synchronization by a factor proportional to $\delta$, in other words,
with perturbations, the transient time towards synchronization is longer.   Moreover,
we have the following: \\
\begin{remark}
If $\widetilde{P}_{ij}$ are nonlinear operators satisfying
$$\Vert \widetilde{P}_{ij}(t,x)\Vert \leq M \Vert x\Vert^{1+c}$$ 
for
some uniform constant $M>0$ and any $c>0$, then $\delta$ can be taken arbitrarily small, that is, 
nonlinear perturbations do not affect the decay rate.
\end{remark} \\

The Theorem \ref{theo:persistence}  ensures that the solutions of Eq. \eqref{eq:main} with initial conditions in an open neighborhood of the synchronization manifold are attracted to it uniformly and exponentially fast. Moreover, it ensures that there are no synchronization loss or bubbling bifurcations
\cite{Rubin,RViana}.

The persistent condition Eq. \eqref{eq:persistentcondition} relates the size of the mismatch to the network structure. We will explore the relation between network structure and $\delta$ in the Corollaries \ref{cor:ER} and \ref{cor:BA}.  
We  relate $\delta$ to the graph structure for two important examples of complex  networks: \\

\begin{itemize}
\item[]--
{\it Homogeneous networks}, where the disparity in the node degrees is small. A paradigmatic example 
is the Erd\"os-R\'enyi (ER) random network: Starting with $n$ nodes the graph is constructed by connecting nodes randomly. 
Each edge is included in the graph with probability $p$ independent from every other edge. If
$p \gg \log n / n$ then all degrees are nearly the same \cite{ER_paper}. \\

\item[]--
{\it Heterogeneous networks}, where a few nodes are massively connected (theses nodes are called \textit{hubs}) while most of the nodes have only a few connections. A typical example of such networks is the Barabási-Albert (BA) random tree. To construct the graph we start with a single edge. Then at each step, we start a new edge from one of the nodes created so far to a new node. The starting node is chosen at random in such a way that the probability to choose a given node is proportional to its degree. \\
\end{itemize} 

Illustrations of Erd\"os-R\'enyi (ER) random networks (homogeneous) and Barab\'asi-Albert (BA) Scale-Free networks (heterogeneous) can be seen in Figure \ref{fig:BA}.
These graphs are random so we want to characterize events in the large network limit.  We say that an event holds asymptotically almost surely if the probability tend to $1$ as $n\rightarrow \infty$. \\

\begin{figure}[!ht]
\centering
\includegraphics[scale=0.45]{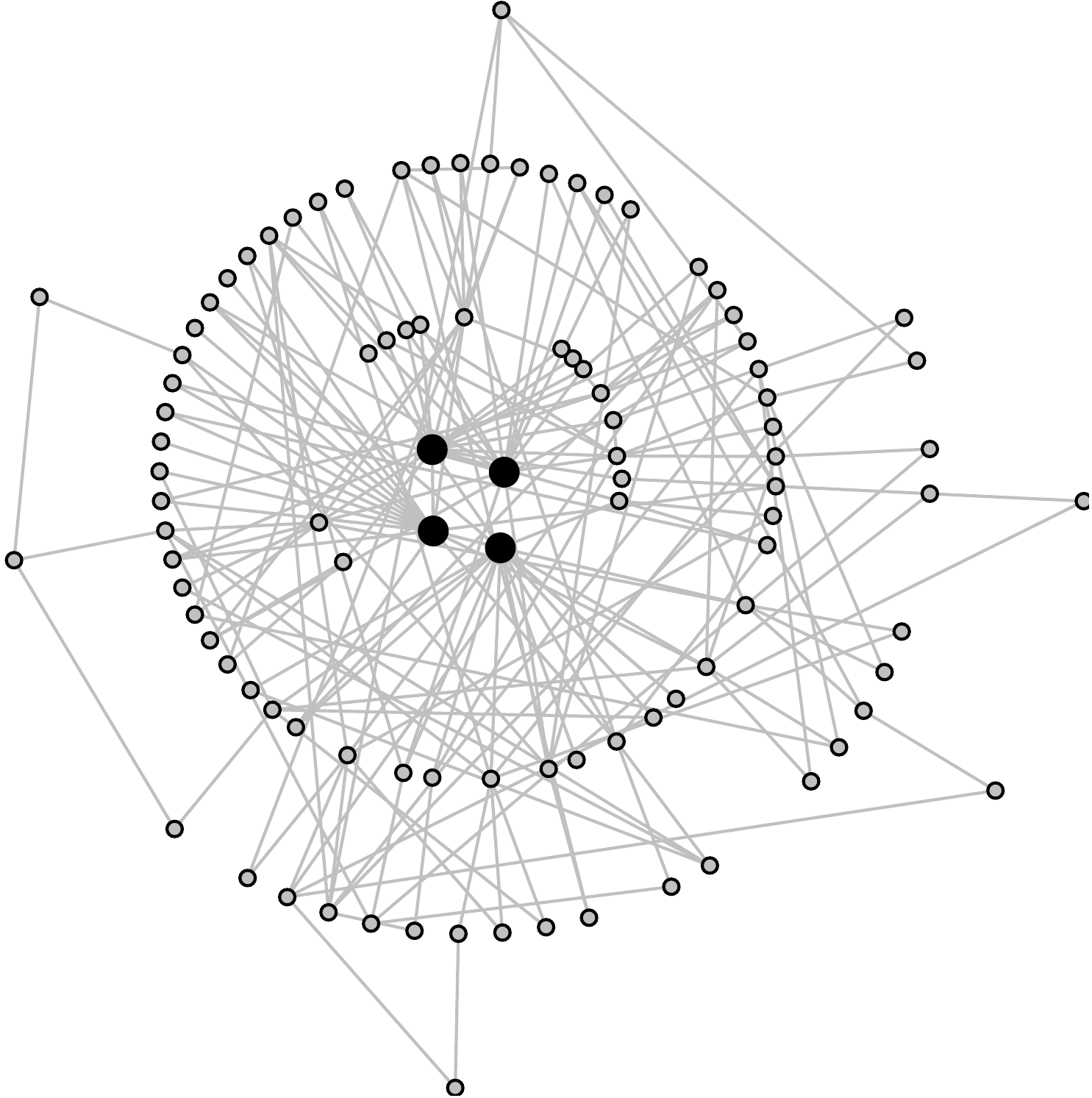}\qquad
\includegraphics[scale=0.3]{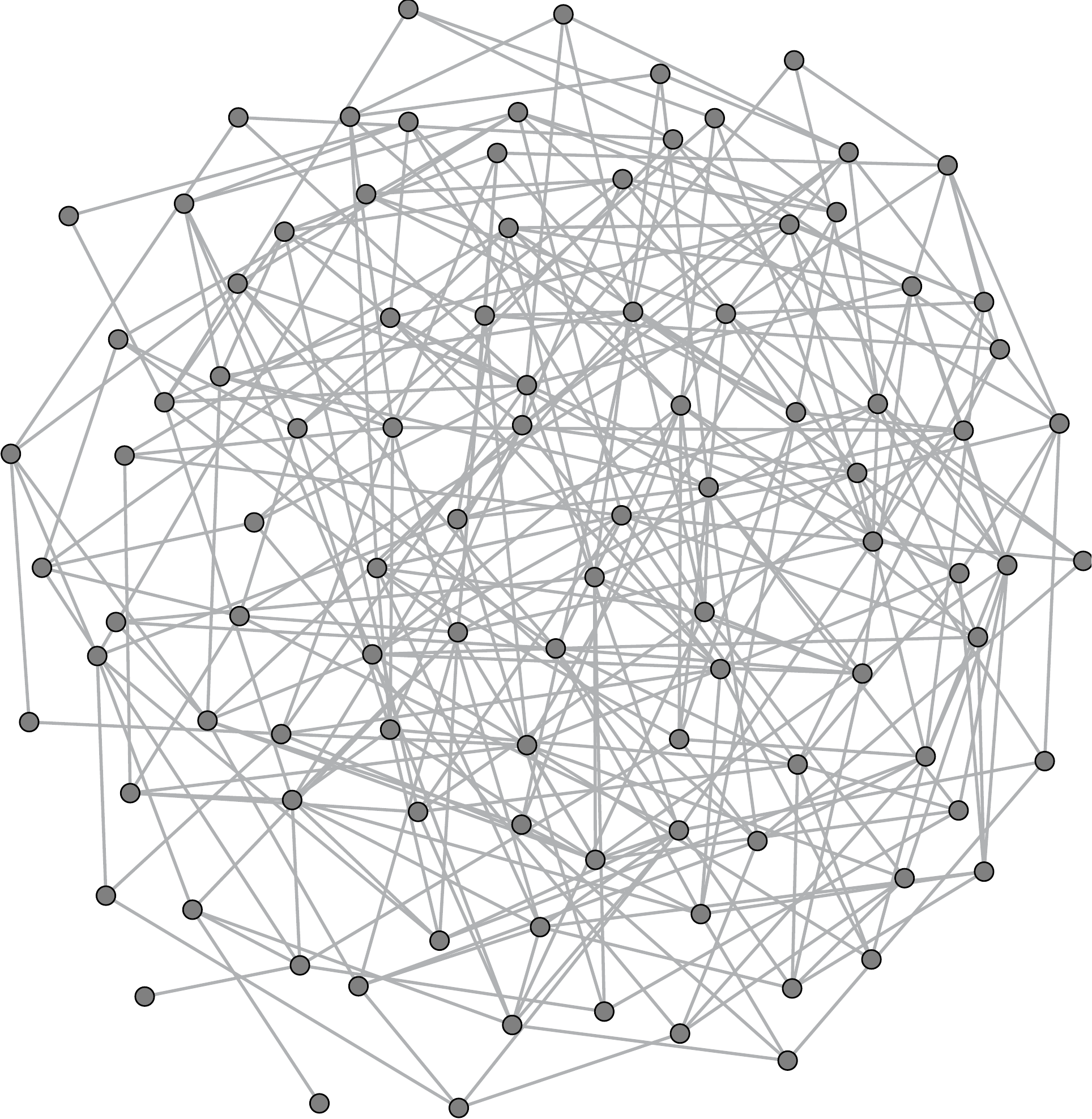}
\caption{Illustrations of a BA network (left) and an ER network (right), both with $n=100$. $BA$ networks have a high heterogeneity in the node's degrees as represented by the four bold nodes. While the mean degree of such BA network is around 4 each of the bold nodes have over 25 connections.  
}
\label{fig:BA}
\end{figure}

\begin{corollary}[ER Networks]\label{cor:ER}
Consider the Theorem \ref{theo:persistence} for an ER network with $p \gg \log n / n$.
Then  asymptotically almost surely exists $K_0 = K_0(\Gamma)$ such that in limit of large coupling parameter $\alpha$ and network size $n$
the perturbation condition \eqref{eq:persistentcondition} reduces to
\begin{equation}
\delta < K_0.
\end{equation}
That is, the perturbation size on ER networks can always be bounded away from zero. Hence, 
ER networks support large mismatches in the coupling function. 
\end{corollary}

In contrast to
homogeneous networks, heterogeneous networks such as BA networks  will support increasingly smaller mismatches in the large $n$-limit. \\

\begin{corollary}[BA Networks]\label{cor:BA}
Consider the Theorem \ref{theo:persistence} for an BA Scale-Free network.
Then  asymptotically almost surely
 there exists $K_1 = K_1(\Gamma)$ such that in limit of large coupling parameter $\alpha$ and network size $n$
the perturbation condition \eqref{eq:persistentcondition} reduces to
\begin{equation}\label{eq_BA_networks}
\delta < {K_1}{n^{-1/2}}.
\end{equation}
\end{corollary}

Now, turning back to Theorem \ref{theo:persistence}, if further information is given on the behavior  of the perturbations $P_{ij}(t)$ the bounds can be improved, that is, even if $\Vert P_{ij}(t)\Vert$ is large, synchronization can be attained. For instance,  consider that the perturbations $P_{ij}(t)$ are taken from a family $\mathcal{P}$ of periodic matrix functions of one parameter $\omega$ (called frequency) and with mean value zero. For the sake of simplicity, lets omit the indexes $ij$ in the next definition. \\
\begin{definition}[Fast Oscillation]\label{def_fast_oscillation}
We say that $P(t) \in \mathcal{P}$ oscillates fast if for any $c,h>0$ there is a
frequency $\omega_0=\omega_0(c,h)$ such that for all $\omega>\omega_0$ then
\begin{equation}\label{eq_fastlimit}
\left\Vert \int_{t_1}^{t_2} P(\omega t)dt \right\Vert \leq c \quad \mbox{for any } t_1 < t_2 < t_1 + h.
\end{equation}
\end{definition}

For this class of perturbation, synchronization is attained and the effect of $P_{ij}(t)$
can be neglected even it is large in magnitude.\\

\begin{theorem}[Fast Limit]\label{theo:fastlimit}
Consider the model in Eq. \ref{eq:main} satisfying Assumptions \ref{assumption_vector_field} 
and \ref{assumption_coupling}. Regardless the values of  $\sup_{t}\Vert P_{ij}(t)\Vert\leq \delta$,
if the perturbations $P_{ij}(t)\in \mathcal{P}$ oscillates fast enough
then for all $\alpha$ satisfying Eq. \eqref{eq:condition1},
the synchronization manifold is locally exponentially stable. Moreover, the decaying rate towards 
synchronization is not affected.
\end{theorem}

\section{Illustrations}\label{Ill}

We present in this section two illustrations for the presented results. The first
illustration  is a simple but rich illustration that
covers Theorems \ref{theo:persistence} and \ref{theo:fastlimit} 
and the second illustration approaches Corollaries \ref{cor:ER} and \ref{cor:BA}.
For both illustrations, we make use of the Lorenz system
 \begin{equation}\label{eq:lorenzsystem}
 \begin{array}{ccc}
 \dot{x}_{i1} & = & 10 (x_{i2}-x_{i1})\\
 \dot{x}_{i2} & = & x_{i1}(28-x_{i3}) -x_{i2}\\
 \dot{x}_{i3} & = & x_{i1}x_{i2} - (8/3)x_{i3}
 \end{array},
 \end{equation}
as the dynamics for the individual node $x_i(t) = (x_{i1}(t),x_{i2}(t),x_{i3}(t))\in\R^3$.

The Lorenz system has an absorbing domain, that is,  there is a compact subset $U\subset \R^3$ to which the solutions of Eq. \eqref{eq:lorenzsystem} will converge \cite{Sparrow}. Therefore, the solutions of this system exist globally and Assumption \ref{assumption_vector_field} follows. Moreover, inside $U$ the system \eqref{eq:lorenzsystem} is chaotic for the chosen parameters \cite{Viana}. If the coupling parameter $\alpha$ is larger than the critical coupling 
Eq. \eqref{eq:condition1} the Lorenz systems will synchronize and have a chaotic dynamics.

For the perturbed coupling
functions we set 
$$
H_{ij}(t,x ) = x +P_{ij}(t)  x,
$$
such that $H_{ij}$ is a perturbation of the identity, where 
\begin{equation}\label{eq_perturbations_operators}
P_{ij}(t) = \delta\cos( t)\operatorname{R}_{ij}, \quad i,j \in \{1,\cdots,n \}
\end{equation}
where $\operatorname{R}_{ij}$ is  random matrix picked independently from  an orthogonal Gaussian ensemble for each $i$ and $j$ normalized according to  $\Vert \operatorname{R}_{ij}\Vert=1$. By construction, $\sup_t\Vert P_{ij}(t)\Vert = \delta$
for every $i$ and $j$, which agrees with Assumption \ref{assumption_coupling} and makes $\delta$ a perturbation control parameter.

We numerically integrate Eq. \eqref{eq:main} using the sixth order Runge-Kutta method with fixed integration step for all illustrations that follows. The initial conditions for each vector state $x_i(0)$ were, also for all experiments that follows, 
$x_i(0)=  (-7,-10,5) + \varepsilon_i$ where
 $\varepsilon_i$ is a random variable in the interval $(0,0.1)$ with a uniform distribution.

\subsection{Synchronization Tongue}

Lets consider the first illustration, namely, when $n=2$ (two coupled oscillators).
Our  Theorem \ref{theo:persistence} provides
a {\it synchronization tongue}.
The persistence condition \eqref{eq:persistentcondition} is of the form
\begin{equation}\label{eq_theoretical_delta}
\delta < c_1 - c_2/\alpha
\end{equation} 
 where $c_1 = \lambda_2\gamma/(K\Vert L\Vert)$ and $c_2 = \eta/(K\Vert L\Vert)$.
 
Considering the Lorenz system \eqref{eq:lorenzsystem} and perturbations according Eq. \eqref{eq_perturbations_operators}
we perform the numerical computation of Eq. \eqref{eq:main} for
combinations of parameters $\alpha$ (coupling) and $\delta$ (perturbation) and compute
the synchronization error $ \Vert x_2(t) - x_1(t)\Vert $. 

We regard the first integration time $\tau =1000$ as a transient and discard it. Then the next $T = 2000$
we compute the mean synchronization error
$$
E(\alpha,\delta) = \frac{1}{T-\tau} \int_{\tau}^T \| x_2(t) - x_1(t) \| dt.
$$
Moreover, for each fixed  $\alpha$ and $\delta$ we average $E$ over ensemble of initial conditions (20 distinct 
initial conditions chosen uniformly as discussed above). We denote this averaged synchronization error by 
$E_a$. We use the triple $(\alpha,\delta,E_a)$ to produce a color map
where the color level represents the synchronization error $E_a$ for the combination $(\alpha,\delta)$.

The color map is depicted in Figure \ref{fig:colormap1} and we call it 
synchronization tongue because of its particular shape.
Note that if $\delta = 0$ (no perturbation at all) we find that
for  $\alpha>0.5$ the Lorenz oscillators synchronize.  
Using the theoretical Equation \eqref{eq_theoretical_delta}  and using the data provided by the numerical simulation one obtain $\delta < 8 - 4/\alpha$. This equation is drawn in Figure \ref{fig:colormap1}
as a yellow solid line.
 
\begin{figure}[!ht]
\centering
\includegraphics[width=10cm,height=7cm]{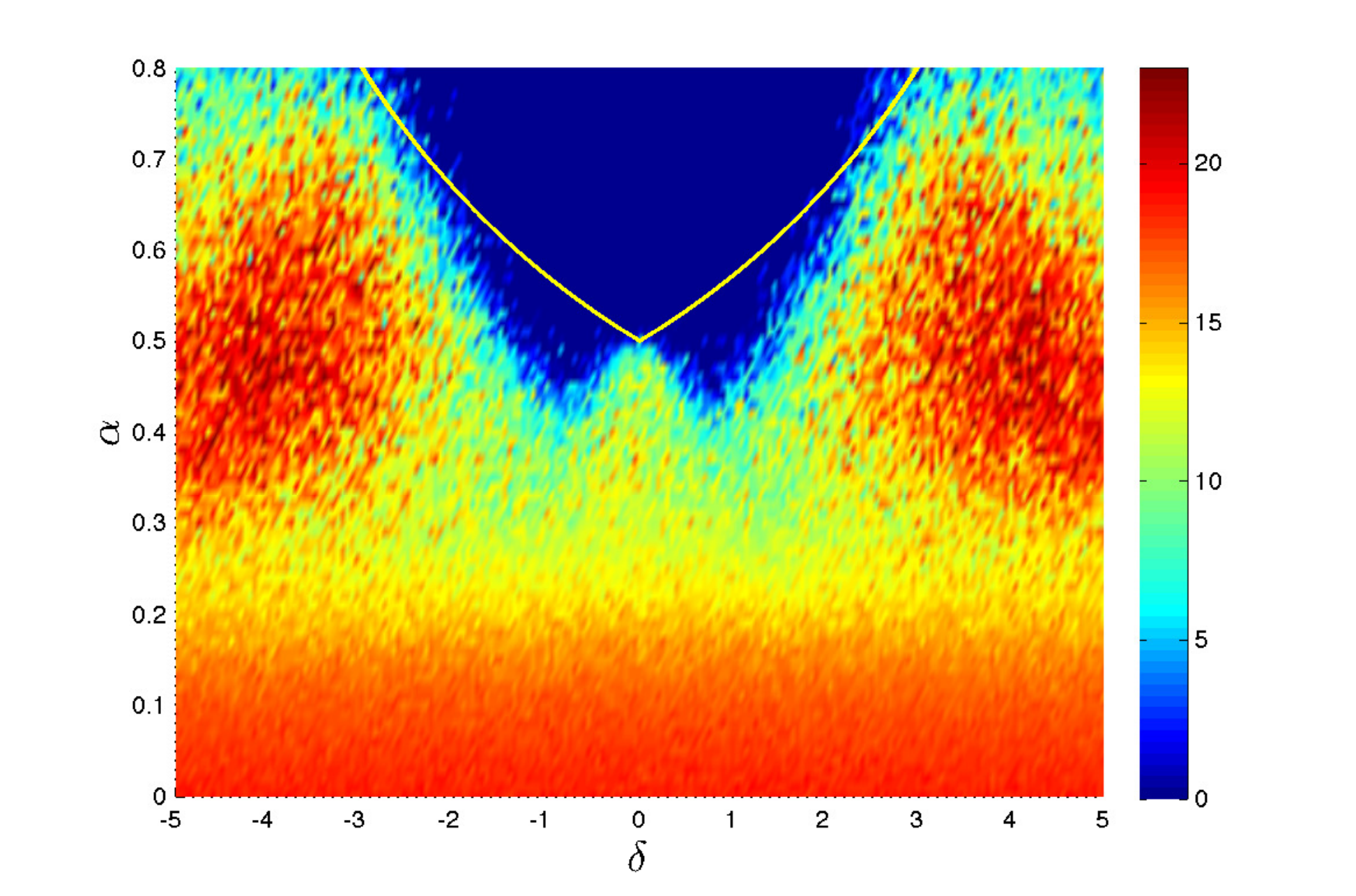}
\caption{Color map for the mean synchronization error.
It has a total of $201(\mbox{horizontal})\times 101(\mbox{vertical})  = 20301$ 
points.
The color scale represents the mean synchronization 
error $E_a$.  The yellow curve represents the theoretical bound
\eqref{eq:persistentcondition} with outcome fitting $\delta = 8 - 4/\alpha$.}\label{fig:colormap1}
\end{figure}

\subsection{ Fast Oscillations}

Let us illustrate the behavior of the synchronization tongue when perturbations oscillates very fast. To this end we 
consider 
\begin{equation}\label{eq_perturbations_operators_w}
P_{ij}(t) = \delta\cos( \omega t)\operatorname{R}_{ij}, \quad i,j \in \{1,\cdots,n \}
\end{equation}
with $\operatorname{R}_{ij}$ chosen as in the previous example. In the limit $\omega \gg1$
the perturbations	
that we are considering (Eq. \eqref{eq_perturbations_operators_w})
fulfills the fast oscillation condition (Definition \ref{def_fast_oscillation}).

Indeed, it easy to compute
$$
\left\Vert \int_{t_1}^{t_2} P_{ij}(\omega t)dt\right\Vert =
 \left\Vert \int_{t_1}^{t_2}  \delta\cos(\omega t) \operatorname{R}_{ij} dt \right\Vert =
|\delta| \left\vert  \frac{ \sin (t_2\omega) - \sin(t_1\omega)}{\omega}\right\vert \leq
\frac{2|\delta|}{\omega}
$$
for any $t_2>t_1$.
So, for any $c>0$ there is an $\omega_0=2|\delta |/c$  so that for every $\omega>\omega_0$
we have $$\left\Vert \int_{t_1}^{t_2} P_{ij}(\omega t)dt\right\Vert \leq c \quad \mbox{ for any } t_2 > t_1.$$

The color map of the Figure \ref{fig:colormap1} was produced using
the perturbations in Eq. \eqref{eq_perturbations_operators_w} with $\omega= 1$.
Now, from Theorem \ref{theo:fastlimit}, we know that in the large limit
of $\omega$ the synchronization tongue in Figure \ref{fig:colormap1}
will flatten at the level $\alpha = 0.5$, which is the level that produces synchronization in a scenario of no perturbations
in coupling function ($\delta=0$).
The Figure \ref{fig:colormaps2} shows the numerical results of this property.

\begin{figure}[!ht]
\centering
\includegraphics[width=10cm,height=7cm]{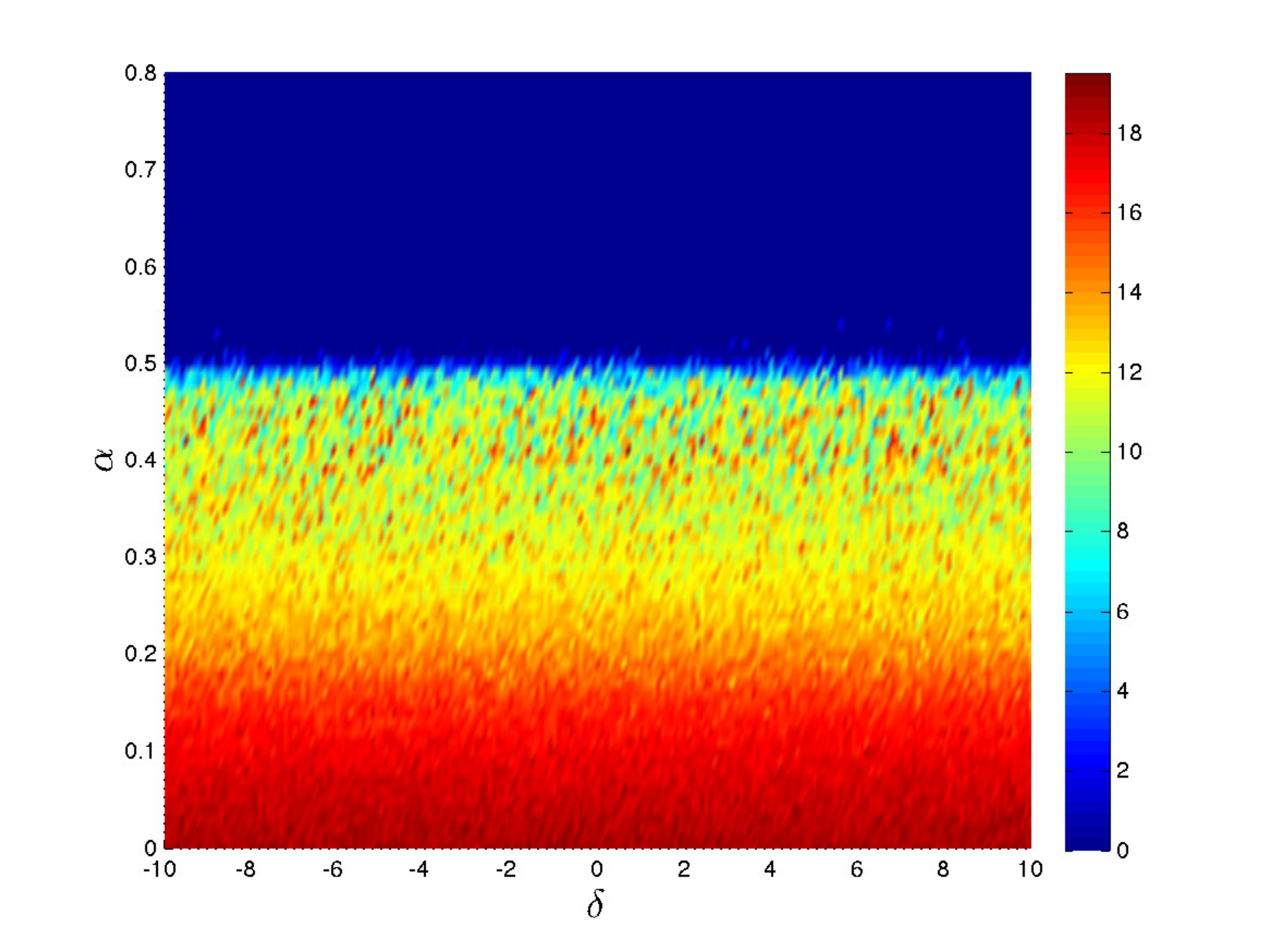}
\caption{Color map for the mean synchronization error for the fast oscillating coupling with $\omega=1000$. 
The synchronization tongue flattens out meaning that coupled oscillators do not feel the oscillating perturbation in the coupling.} 
\label{fig:colormaps2}
\end{figure}

\subsection{Network size effects  on Perturbations}

Corollaries \ref{cor:ER} and \ref{cor:BA} predict interesting system size effects on the perturbations. Here, we wish to illustrate these effects. Hence, we perform numerics experiments considering connected networks with
$n$ nodes, the Lorenz system (Eq. \eqref{eq:lorenzsystem}) as model of isolated dynamics, and coupling
functions accordingly Eq. \eqref{eq_perturbations_operators_w} with $\omega = 1$. 

We determine the effect of the network size $n$ on the perturbation norm $\delta$ as follows. 
For each fixed network size $n$ we start the problem with large coupling $\alpha$ such that the 
system synchronizes at $\delta = 0$ (no perturbation). Then we increase $\delta$ in Eq. \eqref{eq_perturbations_operators_w}. When synchronization is lost at a given $\delta_{\max}$  we stop the simulation.  That is, the value  $\delta_{\max}$ stands for the maximal perturbation value 
that the network synchronization can bare before being destroyed
for any $\delta>\delta_{\max}$. We consider that the synchronization is lost when the synchronization
mean error $E_a > 10$.


Again, we perform  numerical simulations using the sixth order Runge-Kutta method
to evolve the dynamics of Eq. \eqref{eq:main} using $\alpha =5 \gg \eta/(\lambda_2\gamma)$.
For each fixed $n$ we use the fixed
step size $\Delta\delta = 0.01$ to increase the values of $\delta$. For each fixed value of $\delta$
we let the Eq. \eqref{eq:main} evolve for a 
transient time $\tau = 1000$ and then we  compute the synchronization mean error $E_a$ as before.
\newline

\noindent
{\bf ER networks:} We generated ER networks\footnote{We used the software Network Workbench (NWB) to generate all networks used in this paper. NWB is free and it is available at \url{www.nwb.cns.iu.edu}.}
 with fixed probability $p=0.3$ so that the
assumptions in Corollary \ref{cor:ER} holds for $n$ large enough. Our numerical simulations show excellent agreement with  Corollary \ref{cor:ER} -- the effect of the network size $n$ on $\delta$ is nearly constant for
large $n$ and $\delta_{\max}$ is always bounded
always from zero, which can be seen in Figure \ref{fig:ER} where $\log(n) = \log_{10}(n)$.

\begin{figure}[!ht]
\centering
\includegraphics[scale=0.5]{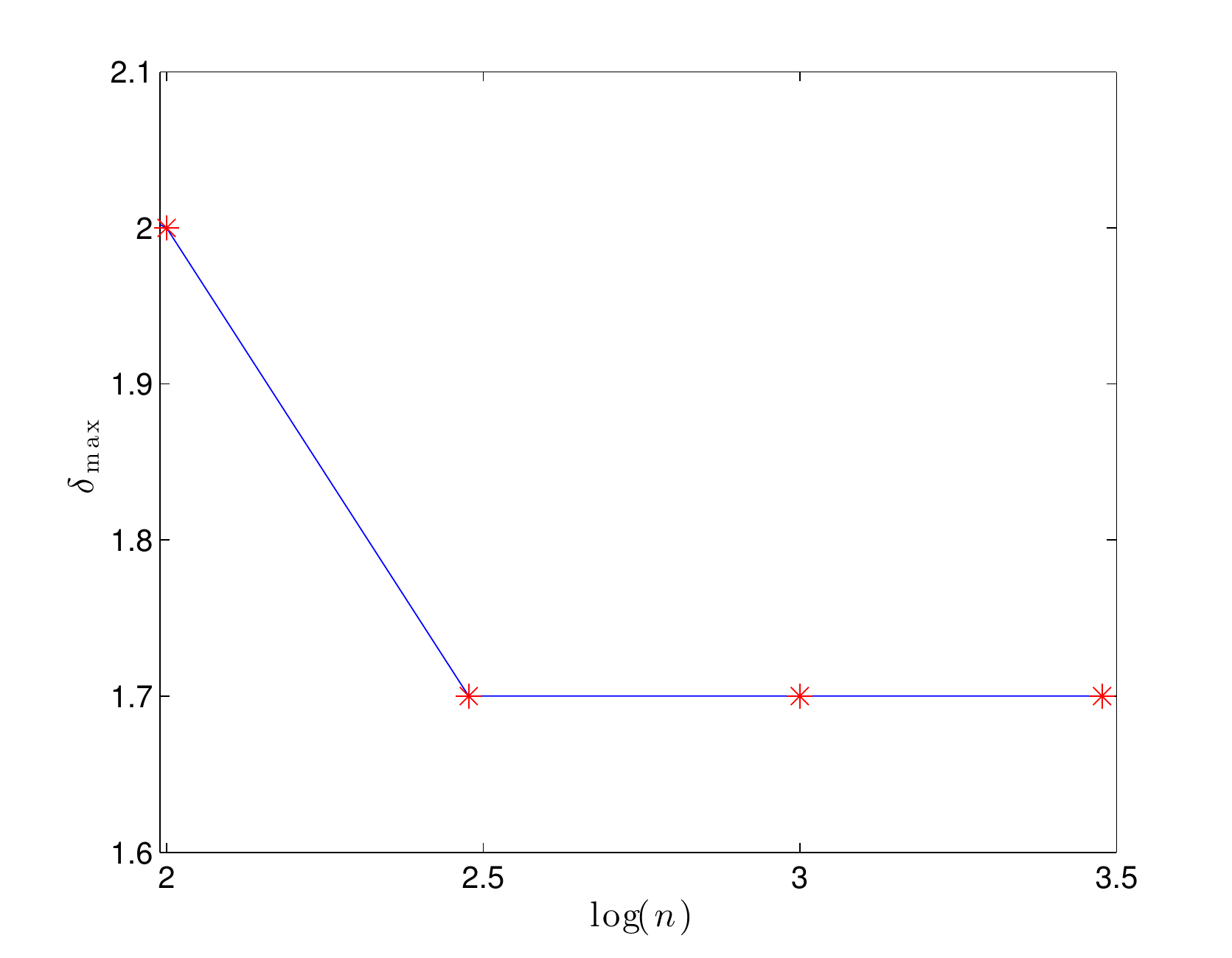}
\caption{Effect of the network size on perturbation size $\delta$ for Erd\"os-R\'enyi random networks. In the large $n$-limit the 
ER network tolerates large mismatches. Moreover, the size of the mismatch does not decrease with the network size as predicted.  The values of $\delta_{\max}$ 
have been rounded to second decimal place.}
\label{fig:ER}
\end{figure}

%
%
%

\noindent
{\bf BA Scale-Free networks:} Corollary \ref{cor:BA} says that, if one perturb the coupling function, it will be hard to synchronize Barab\'asi-Albert Scale-Free networks
in the large limit of $\alpha$ and $n$ because in this case we have $$\delta < O(n^{-1/2}).$$ 
To check this prediction we generated 
BA Scale-Free networks
with parameter
$m_0=2$ (links set by new node) unchanged for each network of size $n$.
The effect of $n$ on $\delta$, as expect from Corollary \ref{cor:BA} can be observed in Figure
\ref{fig_loglog} where again, $\log = \log_{10}$.
\begin{figure}[!ht]
\centering
\includegraphics[scale=0.5]{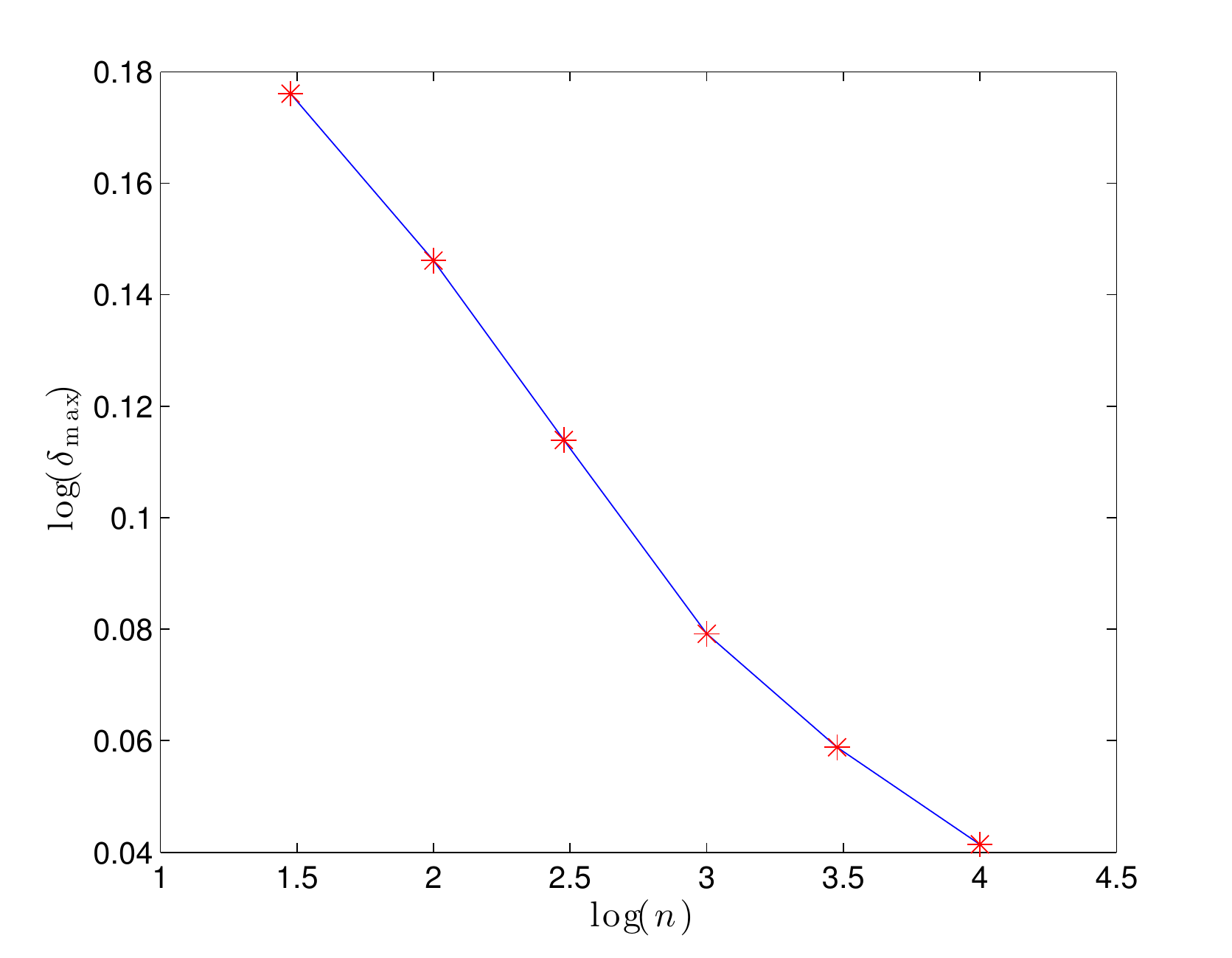}
\caption{Effect of the network size on the perturbation size $\delta$ for BA Scale-Free networks. As predicted in the large $n$-limit BA networks do not tolerate large perturbation. In fact, it forces to perturbation to decay.}
\label{fig_loglog}
\end{figure}

Performing a fitting in the data of the Figure \ref{fig_loglog} we obtain  
$$
\delta_{\max} \propto n^{-\beta}, \mbox{ with }  \beta \approx 0.1.
$$

Our conditions on $\delta$ for the persistence of the network synchronization are sufficient conditions. 
Regarding BA networks, our experiments have showed that the actual decay may be 
slower than the predicted one, that it, slower than $n^{-1/2}$. 

It remains an open question how to obtain conditions that are sufficient and necessary. 
This experiments suggest that the relation between sufficient and necessary conditions may be non-trivial.

\section{Acknowledgments}

DMNM was partially supported by CAPES; TP thanks FAPESP grant 15/08958-4
and EENM thanks FAPESP grant 2011/50151-0 and CNPq.
We are in debt with Paulo R.C. Ruffino for the valuable discussions.

\Appendix

\section{The Proofs}\label{appendix_proofs}

The proof of Theorem \ref{theo:persistence} is given in terms of perturbation theory, 
in particular the roughness of exponential dichotomies by
combining results from Refs. \cite{Pereira2014}  and \cite{coppel}. 
More precisely, firstly we consider the case $\delta=0$ and we use results from \cite{Pereira2014}
to give conditions on $\alpha$. Then we consider the case when $\delta\neq 0$
via the roughness of exponential dichotomies.
The Theorem
\ref{theo:fastlimit} is also a corollary of Theorem \ref{theo:persistence} when
considering the Fast Oscillation Definition \ref{def_fast_oscillation}.

We start with some auxiliary results. The results of the following Lemma \ref{pro_perturbed_node}
are valid for any operator norm.

\begin{lemma}\label{pro_perturbed_node}
Suppose $A(t)$ is a bounded continuous matrix function on an interval $J$ and the evolution operator $\Phi$
of the equation
\begin{equation}\label{eq:non_perturbed_equation}
\dot{X} = A(t)X
\end{equation}
satisfies the inequality
\begin{equation}\label{eq:evolution_operator_non_perturbed_equation}
\Vert \Phi(t,t_0)\Vert \leq K e^{-\nu(t-t_0)} \quad \mbox{for } t\geq t_0.
\end{equation}
If $B(t)$ is a continuous matrix function such that $\Vert B(t) \Vert \leq \theta$ for all $t\in J$ then the
evolution operator $\Psi$ of the perturbed equation
\begin{equation}\label{eq:perturbated_equation}
\dot{Y} = [A(t)+B(t)]Y
\end{equation} 
satisfies the inequality
\begin{equation}\label{eq_inequality_alpha_eta_etc}
\Vert \Psi (t,t_0)\Vert \leq Ke^{\beta(t-t_0)} \quad \mbox{for } t\geq t_0,
\end{equation}
with $\beta = -\nu + \theta K$.
\end{lemma}
The Lemma \ref{pro_perturbed_node} is classical and and its proof can be found \cite{coppel}.

Next, we wish to obtain variational equations for perturbations in a vicinity of the 
synchronization manifold. To this end, we write the solutions of the Eq. \eqref{eq:main} in the block form 
$$X(t) = \operatorname{vec}[x_1(t),x_2(t),\cdots,x_n(t)]\in\R^{nq}$$
where $\operatorname{vec}$ denotes the 
stacking of the columns of $[x_1(t),x_2(t),\cdots,x_n(t)]$ into one long vector \cite{Lancaster1985}.
Likewise,
one can rewrite the whole vector field as $$F(X) = \operatorname{vec}[f(x_1),f(x_2),\cdots,f(x_n)]$$ where
$F:\mathbb{R}^{nq} \to \mathbb{R}^{nq}$. 
Near to the synchronization manifold, one can rewrite the vector solution $X(t)\in\R^{nq}$ as
\begin{equation}\label{eq_X_spanned}
X(t)  = \mathbf{1}\otimes s(t) + \xi(t)
\end{equation}
where the vector $\mathbf{1}=(1,\cdots,1)\in\R^n$
is the eigenvector  of the laplacian matrix $L$ associated with the eigenvalue $0$, $\otimes$ stands
for the Kronecker product, $s(t)$ is the synchronous solution satisfying $\dot{s} = f(s)$
and $\xi(t)=\operatorname{vec}[\xi_i(t),\cdots,\xi_n(t)]\in\R^{nq}$ is a perturbation of the synchronized state. 

Consider the Eq. \eqref{eq:main} with
linearizations in the vector field and coupling function near to the synchronous manifold, that is, writing
$x_i(t) = s(t) + \xi_i(t)$ we have
\begin{equation}\label{eq_perturbed_linearized_model}
 \dot{\xi}_i = Df(s)\xi_i + \alpha\sum_{j=1}^n A_{ij} (\Gamma + P_{ij}(t))(\xi_j-\xi_i) + R(\xi), \quad i=1,\cdots,n
\end{equation}
where $Df(s(t))$ is the Jacobian matrix of the isolated vector field $f$ along the synchronous solution $s(t)$
and $R:\R^{nq}\to\R^q$ is such that $\Vert R(\xi)\Vert = O(\Vert \xi\Vert^2)$ stands for the Taylor remainder of the expansions of the vector field and the coupling function. 
As we are considering
the local stability of the synchronous solution we regard $\Vert R(\xi)\Vert$ being so small that we can neglect it.
Then, putting Eq.  \eqref{eq_perturbed_linearized_model} in the block form, 
the following Lemma \ref{lemma_xi_equation} holds.

\begin{lemma}\label{lemma_xi_equation}
Near to the synchronous manifold the first variational equation of $\xi$ is
\begin{equation}\label{eq_xi}
\dot{\xi} = [\operatorname{I}_{n}\otimes Df(s(t)) - \alpha (L\otimes \Gamma) + \alpha P(t)]\xi
\end{equation}
where $P:\mathbb{R}\times\mathbb{R}^{nq}\to\R^{nq}$ satisfies
\begin{equation}
\Vert P(t)\Vert \leq \Vert L \Vert \delta
\end{equation}
where $\delta$ is given accordingly Assumption \ref{assumption_coupling}.
\end{lemma}

\begin{proof}[Proof of Lemma \ref{lemma_xi_equation}]
We will omit the dependency in $t$ of $P_{ij}(t)$ for the sake of simplicity.
Using the fact that the elements of the laplacian matrix reads $L_{ij} = \delta_{ij}g_i-A_{ij}$, where $\delta_{ij}=1$
if $i=j$ and $0$ otherwise,
the network model \eqref{eq_perturbed_linearized_model} reads
\begin{align}\label{eq_aux_diagonal_terms}
\dot{\xi}_i  & =Df(s)\xi_i + \alpha\sum_{ j=1}^n (\delta_{ij}g_i -L_{ij}) (\Gamma + P_{ij})(\xi_j-\xi_i).
\end{align}
Note that all  diagonal terms $\delta_{ij}g_i$ vanishes because when $i=j$ then $\xi_j-\xi_i=0$. Therefore,
 Eq. \eqref{eq_aux_diagonal_terms} can be written in terms of the laplacian matrix and we have
\begin{align}
\dot{\xi}_i & = Df(s)\xi_i - \alpha\sum_{j=1}^n L_{ij} \Gamma(\xi_j-\xi_i) - \alpha\sum_{ j=1}^n L_{ij} P_{ij}(\xi_j-\xi_i) \nonumber \\
& = Df(s)\xi_i - \underbrace{\alpha\sum_{j=1}^n L_{ij} \Gamma(\xi_j)}_{I} +\underbrace{\alpha\sum_{j=1}^n L_{ij} \Gamma (\xi_i)}_{=0}
- \underbrace{\alpha\sum_{ j=1}^n L_{ij} P_{ij}(\xi_j)}_{II} + \underbrace{\alpha\sum_{ j=1}^n L_{ij} P_{ij}(\xi_i)}_{III}.\label{eq_braces}
\end{align}
In the block form, the each portion $Df(s)\xi_i$ stands for the $i$-th block of 
$[\operatorname{I}_n\otimes Df(s)]\xi$.
It is easy to see that the portion $I$ of the Eq. \eqref{eq_braces} stands for the $i$-th block of
$(L\otimes\Gamma)\xi$. 
For the portion $II$, note that it stands for the $i$-th block of
\begin{equation}\label{eq_portion2}
- \alpha\begin{pmatrix}
L_{11}P_{11} & L_{12}P_{12} & \cdots & L_{1n}P_{1n}\\
L_{21}P_{21} & L_{22}P_{22} & \cdots & L_{2n}P_{2n}\\
\vdots & \vdots & \cdots & \vdots \\
L_{i1}P_{i1} & L_{i2}P_{i2} & \cdots & L_{in}P_{in}\\
\vdots & \vdots & \cdots & \vdots \\
L_{n1}P_{n1} & L_{n2}P_{n2} & \cdots & L_{nn}P_{nn}\\
\end{pmatrix}_{nq\times nq} \xi.
\end{equation}
For the portion $III$, note that it stands for the $i$-th block of
\begin{equation}\label{eq_portion3}
\alpha\begin{pmatrix}
\sum_{j=1}^nL_{1j}P_{1j} & 0 &  \cdots  & 0 \\
0 & \ddots & \cdots & 0\\
\vdots & \vdots & \sum_{j=1}^nL_{ij}P_{ij} & \vdots \\
\vdots & \vdots &  \cdots & \ddots \\
0 & 0 & \cdots & \sum_{j=1}^nL_{nj}P_{nj}
\end{pmatrix}_{nq\times nq} \xi.
\end{equation}

Therefore, adding up the all the portions we end up with the first variational equation for $\xi$ (Eq. \eqref{eq_xi})
where $P(t)$ is a Laplacian-like big perturbation matrix.

Regarding the computation of $\Vert P(t)\Vert$ we have
\begin{align*}
\Vert P(t)\Vert  = 2\max_i \left\Vert \sum_{j=1,j\neq i}^n L_{ij} P_{ij}(t) \right\Vert
 \leq \left( 2\max_i \sum_{j=1,j\neq i}^n  | L_{ij} | \right) \sup_t \Vert P_{ij}(t) \Vert 
 = \Vert L \Vert \delta
\end{align*}
where $\sup_t \Vert P_{ij}(t) \Vert  \leq \delta$ accordingly Assumption \ref{assumption_coupling}.
\qquad
\end{proof}

\subsection{The Proof of Theorem \ref{theo:persistence} (Persistence)}

The aim now is to give conditions on $\alpha$
so that the trivial solution $\xi(t) = {0}$ of Eq. \eqref{eq_xi} is exponentially stable. 
This can be achieved in terms of exponential dichotomies. The case when $P(t) = 0$ was already studied
in Ref. \cite{Pereira2014}.

Now, we split the proof of Theorem \ref{theo:persistence} into two steps. In Step 1, we check that the assumptions
of our Theorem \ref{theo:persistence}
satisfies the hypothesis of Theorem 1 in Ref. \cite{Pereira2014} (when $P(t)=0$) we also discuss the dichotomy parameters and in Step 2 we use the persistence
Lemma \ref{pro_perturbed_node} to conclude the result. 

\textbf{Step 1 :}  ({\it Estimates on Dichotomy parameters})
Lets consider the case with no perturbation on the coupling
function, that is, $P(t) = 0$. 
As we said, this case was already studied in Ref. \cite{Pereira2014}. For completeness we discuss the main steps. Consider the variational equation
$$
\dot{\varphi} = [\operatorname{I}_{n}\otimes Df(s(t)) - \alpha (L\otimes \Gamma)]\varphi
$$
Since, $L$ is undirected it also a diagonal representation $L = R^{-1} \Lambda R$. 
In this setting the change of coordinates
$$
\varphi = [R \otimes I_q]^{-1} \zeta 
$$
block diagonalizes the variational equation
$$
\dot{\zeta} = \bigoplus_{i} (Df(s(t)) - \alpha \lambda_i \Gamma) \zeta,
$$
and since $\varphi$ is not parallel to the synchronization manifold, the eigenvalue $\lambda_1 = 0$ does not 
contribute to the evolution of $\zeta$. In Ref. \cite{Pereira2014} it was shown that under Assumption 1 if one defines 
$$\sigma = \min_{1\leq i\leq q,\, 2\leq j\leq n}\Re (\lambda_j \gamma_i ) >0,$$
and  consider the  
coupling strength threshold given by
\begin{equation}\label{eq_theo1_aux_alpha}
\alpha> \dfrac{\eta}{\sigma}
\end{equation}
then
$$
\| \Phi_{\zeta}(t,t_0)\|_2 \le K e^{-(\alpha \sigma - \eta) (t - t_0)},
$$
where $\Phi_{\zeta}$ is the evolution operator of $\zeta$, $K = K(\Gamma)$ is a constant independent of the network (because of the block structure of the equation) and $\eta = \eta(\Gamma,f)$. 
For the evolution operator for original variables $\varphi$ reads as 
$$
\Phi_{\varphi}(t,t_0) = [ R\otimes I_q]  \Phi_{\zeta}(t,t_0)  [R\otimes I_q]^{-1},
$$
hence
$$
\| \Phi_{\varphi}(t,t_0) \|_2 \le  \kappa_2(R\otimes I_q) \| \Phi_{\zeta}(t,t_0) \|_2,
$$
where $\kappa$ is the condition number. Since $\kappa_2(R\otimes I_q) = \kappa_2(R) $ and 
as $R$ is orthogonal $\kappa_2(R) = 1$ we obtain
$$
\| \Phi_{\varphi}(t,t_0) \|_2 \le  K e^{-(\alpha \sigma - \eta) (t - t_0)},
$$
where $K$ is independent of the network structure. So for every $\alpha$ above the threshold the synchronization manifold is locally exponentially stable.


In our setting we need to check that 
$$\sigma = \min_{1\leq i\leq q,\, 2\leq j\leq n}\Re (\lambda_j \gamma_i ) >0.$$

Note that we are considering only symmetric and connected networks, so the laplacian matrix $L$ itself is symmetric and its eigenvalues
can be ordered as
$0=\lambda_1<\lambda_2\leq \cdots \leq \lambda_n$ and they are real. So,
$$\sigma = \min_{1\leq i\leq q,\, 2\leq j\leq n} \lambda_j \Re (\gamma_i ) = \lambda_2\min_{1\leq i\leq q} \Re (\gamma_i ).$$
Furthermore,
we are considering Assumption \ref{assumption_coupling}, so
$\gamma = \gamma(\Gamma) = \min_{1\leq i\leq q}\Re(\gamma_i)>0$, therefore
$$\sigma = \min_{1\leq i\leq q,\, 2\leq j\leq n}\Re (\gamma_i \lambda_j)  = \lambda_2\gamma >0$$
and the Eq. \eqref{eq_theo1_aux_alpha} translates to
\begin{equation}\label{eq_alpha_thresold}
\alpha >\frac{\eta}{\lambda_2\gamma}.
\end{equation}

Then, the Euclidean norm of the evolution operator $\Phi_{\varphi}$ of the Eq. \eqref{eq_xi}, with $P(t)=0$, reads
\begin{equation}
\Vert \Phi_{\varphi}(t,t_0)\Vert_2 \leq K e^{-(\alpha\lambda_2\gamma - \eta)(t-t_0)}.
\end{equation}

%

\textbf{Step 2:}({\it Persistence})
Considering now the perturbed Eq. \eqref{eq_xi}, note that this equation
has a linear perturbation term $\alpha P(t)$. So, we can
use Lemma \ref{pro_perturbed_node} to study the stability of the synchronous manifold under
this perturbation.

Using Lemma \ref{pro_perturbed_node} and Lemma \ref{lemma_xi_equation} 
we ensure that if $\alpha \sup_t \Vert P(t)\Vert_2 =\theta$
then
there are constants $K >0$ (the same $K$ in Step 1) and $\beta$ such that
the evolution operator $\Psi_{\xi}$ of the perturbed Equation \eqref{eq_xi} reads
$$
\Vert \Psi_{\xi}(t,t_0)\Vert_2 \leq K e^{\beta(t-t_0)}
$$
with
\begin{equation}
\beta = -(\alpha\lambda_2\gamma - \eta) + \theta K.
\end{equation}
In order to guarantee that $\beta<0$ we must have 
\begin{equation}\label{eq_deltaineq}
\theta < \frac{\alpha\lambda_2\gamma -\eta}{K}.
\end{equation}
But, by Lemma \ref{lemma_xi_equation} we have $\sup_t \Vert P(t)\Vert \leq \delta \Vert L\Vert $, then
\begin{align*}
\theta = \alpha \sup_t \Vert  P(t)\Vert_2 & \leq \alpha\sqrt{\sup_t \Vert P(t) \Vert_1 \sup_t \Vert P(t)\Vert}\leq 
\alpha \sqrt{\delta \Vert L\Vert_1 \delta \Vert L\Vert} = \alpha  \delta \sqrt{ \Vert L\Vert_1 \Vert L\Vert}
\end{align*}
where $\Vert \cdot\Vert_1$ stands for the matrix $1$-norm.
As $L$ is symmetric then $\| L\|_1 = \| L\|$ and hence
$\theta \leq \alpha \delta \Vert L\Vert$.

Therefore, the sufficient condition
\begin{equation}\label{eq_delta_threshold}
\delta  < 
\frac{\lambda_2\gamma -\eta/\alpha}{K\Vert L\Vert}.
\end{equation}
ensures that the synchronization manifold is locally exponentially stable.

Note that we always can write the solution of Eq. \eqref{eq_xi}
as $\xi(t) = \Psi_{\xi}(t,t_0)\xi(t_0)$. It implies that
\begin{equation}\label{eq_xi_norm}
\Vert \xi(t)\Vert_2 \leq K e^{\beta(t-t_0)}\Vert \xi(t_0)\Vert_2.
\end{equation}
As we are using the representation $X(t) = \mathbf{1}\otimes s(t) + \xi(t)$, then
$\Vert \xi(t)\Vert_2 = \Vert X(t) - \mathbf{1}\otimes s(t)\Vert_2$. In
a component format the convergence of $\Vert x_j(t) -x_i(t)\Vert_2$
will have the same convergence of
$\Vert X(t) - \mathbf{1}\otimes s(t)\Vert_2$ and then, in Eq. \eqref{eq_xi_norm}, we can replace
$\Vert \xi(t)\Vert_2$ by $\Vert x_j(t) -x_i(t)\Vert_2$  by norm equivalence.Therefore, if one  take $\alpha$ accordingly with Eq. \eqref{eq_alpha_thresold}
and $\delta$ accordingly with Eq. \eqref{eq_delta_threshold}
the synchronization manifold is locally exponentially stable and
the statement of the Theorem \ref{theo:persistence} holds.

\subsection{Proof of Corollary \ref{cor:ER} (ER Networks)}

In the limit of large coupling $\alpha$, Eq. \eqref{eq:persistentcondition} reduces to
\begin{equation}\label{eq_aux1_cor1}
\delta < \frac{\lambda_2\gamma}{K\Vert L\Vert} = \frac{\lambda_2\gamma}{2K\max_i g_i}.
\end{equation}
The algebraic connectivity $\lambda_2$ of an Erd\"os-Rényi random graph
of  $n$ vertices, where the edges are chosen with probability $p = p_0(\log n)/n$
for some constant $p_0>1$ follows 
(see Theorem 1.1 of Ref. \cite{KOLOKOLNIKOV} for more details):
\begin{lemma}[Ref. \cite{KOLOKOLNIKOV}]
Consider a Erd\"os-Rényi random graph on $n$ vertices, where the edges are connected with probability
$$
p = p_0\frac{\log n}{n}
$$
for $p_0>1$ constant in $n$. Then the algebraic connectivity $\lambda_2$ 
is
\begin{equation}\label{eq_aux2_cor1}
\lambda_2 \sim np\left(a(p_0) + O\left(\frac{1}{\sqrt{np}}\right)\right) \quad \mbox{as } n\to\infty
\end{equation}
where $a = a(p_0)\in(0,1)$ denotes the solution of $p_0-1 = ap_0(1-\log a)$ .
\end{lemma}

As a remark, note that $a\to1^{-}$ as $p_0\to\infty$, in other words, if one consider $p_0\gg 1$ then
$a(p_0) \sim 1$. From concentration of degrees (see e.g., the Ref. \cite{FanChung}) we have the following
asymptotic behavior:
\begin{lemma}
Consider a Erd\"os-Rényi graph with $p$ choose as before, then 
\begin{equation}
\| L  \| =  np ( 1+ O(n^{1/2 + \varepsilon}))
\end{equation}
for any $\varepsilon>0$.
\end{lemma}

Therefore, for ER networks with $p_0\gg1$ and in the limit of $n\to\infty$ one have
\begin{equation}\label{eq_aux3_cor1}
\frac{\lambda_2}{\| L \|} =  1 (1 + o(1)) .
\end{equation}
where $o(1)$ stands for the little o notation.
Using Eq. \eqref{eq_aux3_cor1} into \eqref{eq_aux1_cor1} we obtain
\begin{equation}\label{eq_condition_ER}
\delta < \frac{\gamma}{K}= K_0.
\end{equation}
The fact that $\delta$ is always bounded away from zero holds because of the fact that the condition \eqref{eq_condition_ER}
is a sufficient condition. So, for instance, we could take 
$$
0<\frac{K_0}{2} < \delta < K_0
$$
satisfying the statement of Corollary \ref{cor:ER}.

\subsection{Proof of Corollary \ref{cor:BA} (BA Networks)}

Again, in the limit of large coupling $\alpha$, Eq. \eqref{eq:persistentcondition} reduces to
Eq. \eqref{eq_aux1_cor1}. For any connected network, the algebraic connectivity $\lambda_2$
fulfills the following bound:
\begin{lemma}[Ref. \cite{Fiedler1973}]
Let $G$ be an undirected graph. Then the second
smallest eigenvalue $\lambda_2$ of laplacian $L$ satisfies 
\begin{equation}\label{eq_aux1_cor2}
\lambda_2 \leq \left(\frac{n}{n-1}\right) g_{\min}.
\end{equation}
where $g_{min}$ is the minimal degree of the graph.
\end{lemma}

As we are considering BA networks, the minimal degrees equals $m_0$ -- the number of initial edges a nodes is given at each step. 
This number is bounded and independent of the network size $n$. Therefore, for a BA network
$$
\lambda_2 < m
$$
for some constant $m>0$.

The bound on $\| L \|$ follows from Theorem 3.1 of Ref. \cite{Mori}, which we state for completeness. 
\begin{lemma}[Ref. \cite{Mori}]
Write $g_{\max} =  \max_i g_i$. With probability 1 we have
\begin{equation}\label{eq_aux2_cor2}
\lim_{n\to\infty} n^{-1/2} g_{\max}= \mu;
\end{equation}
the limit is almost surely positive and finite, and it has an absolutely continuous distribution.
\end{lemma}

Using Eqs. \eqref{eq_aux1_cor2} and \eqref{eq_aux2_cor2} into Eq. \eqref{eq_aux1_cor1}
we obtain 
$$
\delta < \frac{\gamma \tilde m }{2K\mu n^{1/2}} = {K_1}{n^{-1/2}}
$$
with $K_1 = \gamma \tilde m/(2\mu K) = K_1(\Gamma)$. Where absorbed the dependence of $\alpha$ in the constant $\tilde m$.


\subsection{The Proof of Theorem \ref{theo:fastlimit} (Fast Limit)}

The following Lemma \ref{lemma:perturbation}, adapted from Proposition 6 in Ref. \cite{coppel},
contain almost all the proof of Theorem \ref{theo:fastlimit}.

\begin{lemma}\label{lemma:perturbation}
Consider the Lemma \ref{pro_perturbed_node} and
let $A(t)$ and $B(t)$ be bounded continuous matrix functions where $B(t)$ is periodic with zero mean and oscillates fast
in the sense of Definition \ref{def_fast_oscillation}.

Suppose that the evolution operator $\Phi$ of
\eqref{eq:non_perturbed_equation}
satisfies the inequality \eqref{eq:evolution_operator_non_perturbed_equation}.
Then, there is an $\epsilon > 0$ so that the evolution operator $\Psi$ of the perturbed equation
\eqref{eq:perturbated_equation}
satisfies
\begin{equation}\label{eq_inequalityforfastlimit}
\Vert\Psi(t,t_0) \Vert \leq (1+c)Ke^{\tau(t-t_0)} \quad \mbox{for } t\geq t_0
\end{equation}
where $\tau = -\nu + \epsilon$.
\end{lemma}

The result stated in the Lemma \ref{lemma:perturbation} is valid for any operator norm.
In our case, that is, considering Equation \eqref{eq_xi}, the operators $A(t)$ and $B(t)$ are respectively 
$A(t)  = \operatorname{I}_{n}\otimes Df(s(t)) - \alpha (L\otimes \Gamma)$
and $B(t) = \alpha P(t)$. Due to Assumptions \ref{assumption_vector_field} and \ref{assumption_coupling}
these both operators are bounded for all $t\geq 0$.

From Ref. \cite{coppel} we can see that $$\epsilon = 3KMc + h^{-1}\log[(1+c)K]$$
where $M = \max\{\sup_t \Vert A(t)\Vert_2 ,\sup_t \Vert B(t)\Vert_2 \}$ and $h = t_2-t_1$ as in Definition \ref{def_fast_oscillation}.
Even if $M = \sup_t \Vert \alpha P(t)\Vert_2 \leq \delta \Vert L \Vert$ is large,
we can always make $\epsilon<\nu$ 
if one take $h$ large enough and $c$ small enough and it is always possible because $P(t)$ oscillates fast,
that is, there will always be an $\omega_0 = \omega_0(c,h)$ that satisfies the condition $\epsilon<\nu$ (or $\tau<0$).

Therefore, as $B(t) = \alpha P(t)$ is periodic and oscillates fast, Lemma \ref{lemma:perturbation}
can be applied and the evolution operator $\Psi_{\xi}$ of the Eq. \eqref{eq_xi} satisfies
$$
\Vert \Psi_{\xi}(t,t_0) \Vert_2 \leq (1+c)K e^{\tau(t-t_0)}.
$$
It implies that
\begin{equation}\label{eq_xi_norm_fast}
\Vert \xi(t)\Vert_2 \leq (1+c)K e^{\tau(t-t_0)}\Vert \xi(t_0)\Vert_2.
\end{equation}
and in
a component format the convergence of $\Vert x_j(t) -x_i(t)\Vert_2$
will have the same convergence of
$\Vert X(t) - \mathbf{1}\otimes s(t)\Vert_2 = \Vert \xi(t)\Vert_2$ and then we can replace
$\Vert \xi(t)\Vert_2$ by $\Vert x_j(t) -x_i(t)\Vert_2$  without loss
of generality. That is,  there are
constants $\omega_0 = \omega_0(c,h)>0$, $\rho_0>0$, $K>0$  and $\epsilon=\epsilon(c)>0$ 
such that if $\omega>\omega_0$ and
$\Vert x_j(t_0)-x_i(t_0)\Vert_2\leq \rho_0$,
then
$$
\Vert x_j(t)-x_i(t)\Vert_2\leq (1+c)Ke^{\tau(t-t_0)}\Vert x_j(t_0) - x_i(t_0)\Vert_2
$$
with $\tau = -(\alpha\lambda_2\gamma - \eta) + \epsilon<0$.
Therefore,
the synchronization manifold is locally exponentially stable and the decaying rate towards synchronization is not affected
since we take $\omega>\omega_0$ large enough making $c$ and $\epsilon$ as small as we want.



\end{document}